\documentclass[11pt]{article}
\usepackage{mathrsfs}
\usepackage{amsfonts}
\usepackage{mathrsfs}
\usepackage{amssymb}
\usepackage{color}
\usepackage{latexsym,bm}
\usepackage{amssymb}
\usepackage{amsmath}
\usepackage{amsthm}
\usepackage[dvips]{graphics}
\usepackage[htbp]{graphicx}
\usepackage[all]{xy}
\usepackage{epsfig,latexsym,amssymb, amsmath,amscd,amsthm,amsxtra}
\usepackage{graphicx}
\newtheorem{thm}{Theorem}[section]

\newtheorem{rmk}{Remark}[section]
\newtheorem{lem}{Lemma}[section]
\newtheorem{prop}{Proposition}[section]

\newtheorem{e.g.}[thm]{Example}

\begin{document}
\title{Asymptotic Stability of Solutions for 1-D Compressible Navier-Stokes-Cahn-Hilliard system
}
\author{ Yazhou Chen$^1$
,
 Qiaolin He$^2$
,
 Ming Mei$^{3,4}$, Xiaoding Shi$^{1}$\thanks{\scriptsize{Corresponding author, shixd@mail.buct.edu.cn}}  \\
 \scriptsize{$^{1}$Department of Mathematics, School of Science, Beijing University of Chemical Technology, Beijing, 100029, China}\\
  \scriptsize{$^2$School of Mathematics, Sichuan University, Chengdu, 610064, China}\\
 \scriptsize{$^{3}$Department of Mathematics, Champlain College St.-Lambert, St.-Lambert, Quebec, J4P 3P2, Canada}\\
 \scriptsize{$^{4}$Department of Mathematics and Statistics, McGill University, Montreal, Quebec,H3A 2K6, Canada }}
\date{} \maketitle
\noindent\textbf{Abstract}. This paper is concerned with the evolution of the periodic boundary value problem and the mixed boundary value problem for a compressible  mixture of binary fluids modeled by the  Navier-Stokes-Cahn-Hilliard system in one dimensional space.  The global existence and the large time behavior of the strong solutions for these two systems are studied. The solutions are proved to be asymptotically stable even for the large initial disturbance of the density and the large velocity data. We show that the average concentration difference for the two components of the initial state determines the long time behavior of the diffusive interface for the two-phase flow.

\


\


\section {\normalsize Introduction and Main Result}
\setcounter{equation}{0}
The evolution for the compressible mixture of binary fluids (e.g. foams, solidification processes, fluid-gas interface etc.)  is one of the fundamental problems in hydrodynamical science, and it is attracting more and more attention with the development of modern science and technology.   As is well known, the flow of such fluid can be described by the well-known diffusive interface system of equations nonlinearly coupled the Navier-Stokes and Cahn-Hilliard equations, called Navier-Stokes-Cahn-Hilliard system \cite{AF2008,DingLi2013,LT-1998} as follows:
\begin{equation}\label{3d-NSCH}
\left\{\begin{array}{llll}
\displaystyle \partial_t\rho+\textrm{div}(\rho \mathbf{u})=0,  \\
\displaystyle \partial_t(\rho \mathbf{u})+\textrm{div}(\rho \mathbf{u}\otimes \mathbf{u})=\textrm{div}\mathbb{T},\\
\displaystyle\partial_t\big(\rho\chi\big)+\mathrm{div}(\rho \chi \mathbf{u})=\Delta\mu,\\
\displaystyle\rho\mu-\frac{1}{\epsilon}\rho\frac{\partial f}{\partial \chi}+\epsilon \Delta\chi\in\partial I(\chi),
\end{array}\right.
\end{equation}
where  the unknown functions $\rho$, $\mathbf{u}$, $\mu$ and $\chi$ denote the total density, the mean velocity,   the chemical potential and the
  difference of the two components for the fluid mixture respectively.
More precisely, the  difference of two components  for the fluid mixture is $\chi=\chi_1-\chi_2$, where $\chi_i=\frac{M_i}{M}$, and $M_i$  denotes the mass concentration of the fluid $i~(i=1,2)$ and  mass of the components in the representative material volume $V$ respectively. The total density is given by $\rho=\rho_1+\rho_2$, where $\rho_i=\frac{M_i}{V}$ denotes the apparent mass density of the fluid $i$. The mean velocity $\mathbf{u}$ is defined as the average velocity given by $\rho \mathbf{u}=\rho_1\mathbf{u}+\rho_2 \mathbf{u}$. $\partial I(\cdot)$ is the subdifferential of the indicator function $I(\cdot)$ of the set $[-1,1]$. The Cauchy stress-tensor is represented by
\begin{equation}\label{Cauchy stress-tensor}
  \mathbb{T}=\mathbb{S}-\epsilon\big(\nabla\chi\otimes\nabla\chi-\frac{|\nabla\chi|^2}{2}\mathbb{I}\big),
\end{equation}
 and the conventional Newtonian viscous stress is
\begin{equation}\label{conventional Newtonian viscous stress}
  \mathbb{S}=\nu(\chi)\big((\nabla \mathbf{u}+\nabla^{\top}\mathbf{u})-\frac{2}{3}\mathrm{div}\mathbf{u}\mathbb{I}\big)-p\mathbb{I}+\eta(\chi)\mathrm{div}\mathbf{u}\mathbb{I},
\end{equation}
 where $\mathbb{I}$ is the unit matrix, $\nu(\chi)>0, \ \eta(\chi)\geq0$ are defined as viscosity coefficients, $\epsilon>0$ is the thickness of the diffuse interface of the fluid mixture, and
 $f=f(\chi)$ is the potential free energy density. We suppose that $f$ satisfies the Ginzburg-Landau double-well potential non-smooth model which is widely used in  \cite{BSSW2012, BE1991, GW2012, HGW2011, OP1988, WX2017,Yin1994, Zheng1986} and the references therein:
\begin{equation}\label{the formula of free energy density}
f(\chi)=\left\{\begin{array}{llll}
 \displaystyle\frac{1}{4}\big(\chi^2-1\big)^2,\quad & \mathrm{if}\ -1\leq\chi\leq1,\\
 \displaystyle+\infty,\quad & \mathrm{if}\ |\chi|>1.
 \end{array}\right.
\end{equation}
Physically, the pressure $p$ is given by
\begin{equation}\label{the formula of pressure}
   p=a\rho^{\gamma},\ \ \gamma\geq1,
\end{equation}
where $a$ is a positive constant,  $\gamma\geq1$ is the adiabatic constant.
\begin{rmk} Compared to the  momentum equation in compressible Navier-Stokes system, An additional force $-\epsilon\big(\nabla\chi\otimes\nabla\chi-\frac{|\nabla\chi|^2}{2}\mathbb{I}\big)$ is added  in the Cauchy stress tensor, which describes the capillary effect associated with free energy
\begin{equation}\label{simplified total energy}
   E_{\mathrm{free}}(\rho,\chi)=\int_\Omega\big(\frac{\rho }{\epsilon} f(\chi)+\frac{\epsilon}{2}|\nabla\chi|^2\big)dx.
\end{equation}
To avoid the difficulty of the estimation of density gradient, \eqref{simplified total energy} was first presented by Abels-Feireisl \cite{AF2008} which is a simplified model of the following total free energy \begin{equation}\label{original total energy}
   \tilde{E}_{\mathrm{free}}(\rho,\chi)=\int_\Omega\big(\frac{\rho }{\epsilon} f(\chi)+\frac{\rho\epsilon}{2}|\nabla\chi|^2\big)dx,\notag
\end{equation}
proposed by Lowengrub and Truskinovsky in \cite{LT-1998}.
The potential \eqref{the formula of free energy density} is the polynomial non-smooth approximation of the so-called logarithmic potential  suggested by  Cahn-Hilliard \cite{CH1958}
\begin{equation}\label{the logarithmic potential}
 f(\chi)=\frac{1}{2}\theta\Big((1-\chi)\ln(\frac{1-\chi}{2})+(1+\chi)\ln(\frac{1+\chi}{2})\Big)-\frac{\theta_c}{2}\chi^2,\notag
\end{equation}
where $\theta$ and $\theta_c>0$ are positive constants. It follows that $f$ is  convex on $(-1,1)$ for $\theta>\theta_c$, and
has the bistable form for $\theta<\theta_c$.
\end{rmk}
Now let us draw the background  on the well-posedness of the solutions for the system \eqref{3d-NSCH}. We begin by recalling the compressible Navier-Stokes system.
In 1980, Matsumura-Nishida \cite{MN1980} first proved the global existence of smooth solutions for Navier-Stokes equations    for initial data close to a non-vacuum equilibrium. Later, Kawashima-Matsumura \cite{KM1985}, Matsumura-Nishihara \cite{MN1985}-\cite{MN1986} obtained the asymptotic stability for the rarefaction waves and shock waves  in one-dimensional successively. We also refer to the papers due
to Huang-Li-Matsumura \cite{HLM2010}, Huang-Matsumura \cite{HM2009}, Huang-Matsumura-Xin \cite{HMX2006}, Huang-Wang-Wang-Yang \cite{HWWY2015}, Shi-Yong-Zhang \cite{SYZ2016}-\cite{SYZ2017} and the references therein, in  which the asymptotic stability on  the superposition of nonlinear waves are studied.

In the case of diffusive interface model for two immiscible fluids,
in order  to understand the motion of the interfaces between immiscible fluids, phase field models for mixtures of two constituents were studied over
the past century which can be traced to van der Waals \cite{V1894}.  Cahn-Hilliard \cite{CH1958} first replaced the sharp interface into
the narrow transition layer across which the fluids may mix, and the famous Cahn-Hilliard equations was proposed to describe these diffusive
 interfaces between two immiscible fluids,  Lowengrub-Truskinovsky \cite{LT-1998}  added the effect of the motion of the particles
 and the interaction with the diffusion into Cahn-Hilliard quation, and put forward the Navier-Stokes-Cahn-Hilliard equations.
 The similar result was established by Abels-Feireisl \cite{AF2008},
 Anderson-Mcfadden-Wheeler \cite{AMW1998} and the references therein. Heida-M$\mathrm{\mathrm{\acute{a}}}$lek-Rajagopal \cite{HMR2012} and Kotschote \cite{Kotschote2016} generalized these models to non-isentropic case.
The global existence of weak solutions in three dimensional was  obtained by Abels-Feireisl \cite{AF2008},
 the  method they used is the framework  which was introduced by Leray \cite{Leray1934} and Lions \cite{Lions1998}. Kotschote-Zacher \cite{KZ2015} established a local existence and uniqueness result for strong solutions. Ding-Li \cite{DingLi2013} proved the global existence of strong solution  in a bounded domain with initial boundary condition \eqref{mixed Boundary} in one dimension.

In this paper, on the basis of  the previous work outlined above, we begin to study the large time behavior for the solutions of the system \eqref{3d-NSCH} in $\Omega\times(0,+\infty)$
\begin{equation}\label{Euler-NSCH}
\left\{\begin{array}{llll}
\displaystyle \rho_t+(\rho u)_x=0,  \\
\displaystyle \rho u_t+\rho uu_x+p_x=\nu u_{xx}-\frac{\epsilon}{2}\big(\chi_x^2\big)_x,\\
\displaystyle\rho\chi_t+\rho u\chi_x=\mu_{xx},\\
\displaystyle  \rho\mu-\frac{\rho}{\epsilon}(\chi^3-\chi)+\epsilon\chi_{xx}\in\partial I(\chi),
\end{array}\right. \end{equation}
where $\Omega\subseteq\mathbb{R}, t\in(0,+\infty)$. We study only two model cases for \eqref{Euler-NSCH}:

\begin{enumerate}
\item[(i)]  $L$-periodic  boundary case, that is $\Omega=\mathbb{R}$, and
\begin{equation}\label{periodic boundary}
\left\{\begin{array}{llll}
(\rho,u,\chi)(x,t)=(\rho,u,\chi)(x+L,t),\ \ & x\in\mathbb{R},t>0,\\
(\rho,u,\chi)\big|_{t=0}=(\rho_0,u_0,\chi_0),\ \ & x\in\mathbb{R}.
\end{array}\right.
\end{equation}

\item[(ii)] Mixed  boundary case, that is $\Omega=[0,L]$, and
\begin{equation}\label{mixed Boundary}
\left\{\begin{array}{llll}
\big(u,\chi_x,\mu_x\big)\big|_{x=0,L}=(0,0,0),\ \ &t\geq0,\\
(\rho,u,\chi)\big|_{t=0}=(\rho_0,u_0,\chi_0),\ \ & x\in[0,L].
\end{array}\right.
\end{equation}
\end{enumerate}

We denote by $C$ and $c$ the positive generic constants  without confusion throughout this paper.
  $L^2$ denotes the space of Lebesgue measurable functions on $\mathbb{R}$ which are square integrable, with the norm $\|f\|=(\int_{0}^L|f|^2)^{\frac{1}{2}}$.
 $H^l(l\geq0)$ denotes the Sobolev space of $L^2$-functions $f$ on $\mathbb{R}$ whose derivatives $\partial^j_x f,  j=1,\cdots,l$ are $L^2$
 functions too, with the norm
$  \|f\|_l=(\sum_{j=0}^l\|\partial^j_x f\|^2)^{\frac{1}{2}}$.
Let
\begin{equation}\label{average value}
 \bar \rho=\frac{1}{L}\int_0^L\rho_0dx, \quad \overline{\rho u}=\frac{1}{L}\int_0^L\rho_0 u_0dx, \quad \overline{\rho\chi}=\frac{1}{L}\int_0^L\rho_0\chi_0dx,\quad  \bar u=\frac{\overline{\rho u}}{\bar\rho}, \quad  \bar\chi=\frac{\overline{\rho\chi}}{\bar\rho},
\end{equation}
and then, for the both cases above, we have
\begin{equation}\label{average}
 \int_0^L(\rho-\bar\rho)dy=0,\ \int_0^L(\rho u-\overline{\rho u})dy=0,\ \int_0^L(\rho\chi-\overline{\rho\chi})dy=0.
\end{equation}
For the periodic boundary condition (i), because the periodic solutions $(\rho,u,\chi)(x,t)$ of \eqref{Euler-NSCH} with the period $L$ in the whole space $\mathbb{R}$ can
be regarded as $L$-periodic extensions of that on $[0,L]$, we only need to consider
the system \eqref{Euler-NSCH} on the bounded interval $[0,L]$. We introduce the Hilbert space $L^2_{\mathrm{per}}(\mathbb{R})$ of locally square integrable
functions that are periodic with the period $L$:
\begin{equation}\label{periodic function space}
L^2_{\mathrm{per}}(\mathbb{R})=\Big\{g(x)\big|g(x+L)=g(x)\ \mathrm{for\ all}\ x\in\mathbb{R},\ {\mathrm{and}\ } g(x)\in L^2(0,L) \Big\}
\end{equation}
with the norm denoted also by $\|\cdot\|$ (without confusion) which is given by the integral over $[0,L]$, $\|g\|=(\int_0^L|g(x)|^2dx )^{\frac{1}{2}}$.
$H_{\mathrm{per}}^l(\mathbb{R}) \ (l\geq0)$ denotes the Sobolev space of $L_{\mathrm{per}}^2(\mathbb{R})$-functions $g$ on $\mathbb{R}$ whose derivatives $\partial^j_x g,  j=1,\cdots,l$ are $L_{\mathrm{per}}^2$
 functions too, with the norm
$  \|g\|_l=(\sum_{j=0}^l\|\partial^j_x g\|^2)^{\frac{1}{2}}$.
Throughout this paper, the initial and boundary data for the density, velocity and concentration difference of two components are assumed to be:
\begin{equation}\label{initial data of density}
\left.\begin{array}{llll}
 \qquad(\rho_0,u_0)\in H_{\mathrm{per}}^2,\ (\mathrm{periodic\ boundary\ case\  (i)}),\\
 \mathrm{\ or\ } \rho_0\in H^2([0,L]), u_0\in H^1_0([0,L])\cap H^2([0,L]),\ (\mathrm{mixed\ boundary\  case\  (ii)}),
 \end{array}\right.
\end{equation}
\begin{equation}\label{initial data}
\left.\begin{array}{llll}
  \quad\chi_0\in H_{\mathrm{per}}^4,\ (\mathrm{periodic\  case\  (i)}),\\
  \mathrm{or\ }\chi_0\in H^4([0,L]),\ (\mathrm{mixed\ boundary\  case\  (i)}),
\end{array}\right.
\end{equation}
\begin{equation}\label{initial data-1}
 \inf_{x\in[0,L]}\rho_0>0,\  \chi_0\in[-1,1],
\end{equation}
\begin{equation}\label{Compatibility condition rho}
\rho_{t}(x,0)=-\rho_{0x}u_0-\rho_0u_{0x},
\end{equation}
\begin{equation}\label{Compatibility condition u}
u_{t}(x,0)=\frac{\nu}{\rho_0}u_{0xx}-u_0u_{0x}-\frac{p'_\rho(\rho_0)}{\rho_0}\rho_{0x}-\frac{\epsilon}{\rho_{0x}}\chi_{0x}\chi_{0xx},
\end{equation}
\begin{eqnarray}\label{Compatibility condition chi}
\chi_{t}(x,0)&=&-u_0\chi_{0x}-\frac{\epsilon\chi_{0xxxx}}{\rho_0^2}+\frac{2\epsilon\rho_{0x}}{\rho_0^3}\chi_{xxx}+\epsilon\Big(\frac{\rho_{0xx}}{\rho_0^3}
-\frac{2\rho_{0x}^2}{\rho_0^4}\Big)\chi_{0xx}\notag\\
&&+\frac{(3\chi_0^2-1)\chi_{0xx}}{\epsilon\rho_0}
+\frac{6\chi_0\chi_{0x}^2}{\epsilon\rho_0}.
\end{eqnarray}
 Our main result is that  we show  the average concentration difference for the two components of the initial state determines the long time behavior of the diffusive interface for the two-phase flow. This leads to a separation of two regions  for the initial concentration difference: regions $A_{\textrm{stable}}$ and $A_{\textrm{unstable}}$, which correspond to the  stable and  unstable regions.
\begin{equation}\label{stable interval and unstable interval}
  A_{\textrm{stable}}=(-\infty,\frac{\sqrt{3}}{3})\cup(\frac{\sqrt{3}}{3},+\infty),\ \ A_{\textrm{unstable}}=(-\frac{\sqrt{3}}{3},\frac{\sqrt{3}}{3}).
\end{equation}
\begin{rmk}
We rewrite  \eqref{Euler-NSCH}$_{3,4}$ as follows
\begin{eqnarray}\label{Cahn-Hilliard equation}
\chi_{t}+u\chi_{x}+\frac{\epsilon\chi_{xxxx}}{\rho^2}-\frac{2\epsilon\rho_{x}}{\rho^3}\chi_{xxx}
-\epsilon\Big(\frac{\rho_{xx}}{\rho^3}-\frac{2\rho_{x}^2}{\rho^4}\Big)\chi_{xx}
-\frac{3\chi^2-1}{\epsilon\rho}\chi_{xx}
-\frac{6}{\epsilon\rho}\chi\chi_{x}^2\in\partial I(\chi),
\end{eqnarray}
It should be pointed out that the reason for the classification of stable and unstable regions is just that  the energy estimates(please refer to \eqref{the basic energy equality-2 for  concentration difference} and \eqref{the basic energy inequality-2 for density velocity and concentration difference}) ask for the coefficient $\frac{3\chi^2-1}{\epsilon\rho}$ in parabolic equation \eqref{Cahn-Hilliard equation}  should be positive. If $\chi\in A_{\mathrm{unstable}}$, it is thus physically unstable and mathematically ill-posed for the
two-phase flow, and the phase separation will occur in this region.
\end{rmk}
Now we will give the global existence and asymptotic stability results  as follows.

 \begin{thm}[Periodic boundary case]  \label{main thm-1}
Assume that $(\rho_0,u_0,\chi_0)$ satisfies   the periodic boundary \eqref{periodic boundary} with the regularities given in \eqref{initial data of density}-\eqref{Compatibility condition chi},  then there exists a unique global strong solution $(\rho,u,\chi)$ of the system \eqref{Euler-NSCH}-\eqref{periodic boundary}, satisfying that, for any $T>0$, there exists positive constants $m,M>0$ such that
\begin{eqnarray}\label{global solution for periodic boundary}
&&\rho\in L^\infty(0,T,H_{\mathrm{per}}^2),\rho_t\in L^\infty(0,T;H_{\mathrm{per}}^1),\notag\\
&&u\in L^\infty(0,T;H_{\mathrm{per}}^2)\cap L^2(0,T;H_{\mathrm{per}}^3), u_t\in L^\infty(0,T;H_{\mathrm{per}}^1),\notag\\
&&\chi\in L^\infty(0,T;H_{\mathrm{per}}^4), \chi_t\in L^\infty(0,T;L_{\mathrm{per}}^2)\cap L^2(0,T;H_{\mathrm{per}}^2),\\
&&\mu\in L^\infty(0,T;H_{\mathrm{per}}^2)\cap L^2(0,T;H_{\mathrm{per}}^4),\ \mu_t\in L^2(0,T;L_{\mathrm{per}}^2),\notag\\
 &&  \chi\in[-1,1],\ m\leq\rho\leq M,\ \mathrm{for \ all}\ (x,t)\in\mathbb{R}\times[0,T].\notag
\end{eqnarray}
 Furthermore, if $\bar{\chi}\in A_{\textrm{stable}}\cap[-1,1]$, then there exists a small positive constant $\varepsilon_1$, such that if
\begin{equation}\label{small condition for periodic boundary}
\|\chi_0-\bar\chi\|_4^2\leq \varepsilon_1,
\end{equation}
then the  system \eqref{Euler-NSCH}-\eqref{periodic boundary} admits a unique  global  solution $(\rho,u,\chi)(x,t)$ in $\mathbb{R}\times[0,\infty)$ which satisfies
\begin{eqnarray}\label{asymptotic behavior for periodic boundary}
&&\rho\in L^\infty(0,+\infty,H_{\mathrm{per}}^2),\rho_t\in L^\infty(0,+\infty;H_{\mathrm{per}}^1),\notag\\
&&u\in L^\infty(0,+\infty;H_{\mathrm{per}}^2)\cap L^2(0,+\infty;H_{\mathrm{per}}^3), u_t\in L^\infty(0,+\infty;H_{\mathrm{per}}^1),\notag\\
&&\chi\in L^\infty(0,+\infty;H_{\mathrm{per}}^4), \chi_t\in L^\infty(0,+\infty;L_{\mathrm{per}}^2)\cap L^2(0,+\infty);H_{\mathrm{per}}^2),\\
&&\mu\in L^\infty(0,+\infty;H_{\mathrm{per}}^2)\cap L^2(0,+\infty;H_{\mathrm{per}}^4),\ \mu_t\in L^2(0,+\infty;L_{\mathrm{per}}^2),\notag\\
 &&  \chi\in[-1,1],\ \ \mathrm{for \ all}\ (x,t)\in\mathbb{R}\times[0,+\infty), \notag
\end{eqnarray}
and moreover
\begin{equation}\label{asymptotic stability}
\lim_{t\rightarrow\infty}\sup_{x\in\mathbb{R}}\big|(\rho-\bar\rho,u-\bar u,\chi-\bar\chi)\big|=0.
\end{equation}
\end{thm}

For the mixed initial boundary  value problem,  we have the following results concerning the asymptotic behavior of the global classical solutions.

 \begin{thm}[Mixed boundary case]  \label{main thm-2}
Assume that $(\rho_0,u_0,\chi_0)$ satisfies  the mixed boundary \eqref{mixed Boundary} with the regularities given in \eqref{initial data of density}-\eqref{Compatibility condition chi},  then
there exists a unique global strong solution $(\rho,u,\chi)$ of the system \eqref{Euler-NSCH},\eqref{mixed Boundary},
 satisfying that, for any $T>0$, there exists positive constants $m,M>0$ such that
\begin{eqnarray}\label{global solution for mixed boundary}
&&\rho\in L^\infty(0,T,H^2),\rho_t\in L^\infty(0,T;H^1),\notag\\
&&u\in L^\infty(0,T;H_0^1\cap H^2)\cap L^2(0,T;H^3), u_t\in L^\infty(0,T;H^1),\notag\\
&&\chi\in L^\infty(0,T;H^4), \chi_t\in L^\infty(0,T;L^2)\cap L^2(0,T;H^2),\\
&&\mu\in L^\infty(0,T;H^2)\cap L^2(0,T;H^4),\ \mu_t\in L^2(0,T;L^2),\notag\\
 &&  \chi\in[-1,1],\ m\leq\rho\leq M,\ \mathrm{for \ all}\ (x,t)\in[0,L]\times[0,T].\notag
\end{eqnarray}
 Furthermore, if $\bar{\chi}\in A_{\textrm{stable}}\cap[-1,1]$, then there exists a small positive constant $\varepsilon_1$, such that if
\begin{equation}\label{small condition}
\|\chi_0-\bar\chi\|_4^2\leq \varepsilon_1,
\end{equation}
then the  system \eqref{Euler-NSCH}-\eqref{mixed Boundary} admits a unique  global  solution $(\rho,u,\chi)(x,t)$ in $[0,L]\times[0,\infty)$ which satisfies
\begin{eqnarray}\label{rho u chi for mixed boundary}
&&\rho\in L^\infty(0,+\infty,H^2),\rho_t\in L^\infty(0,+\infty;H^1),\notag\\
&&u\in L^\infty(0,+\infty;H_0^1\cap H^2)\cap L^2(0,+\infty;H^3), u_t\in L^\infty(0,+\infty;H^1),\notag\\
&&\chi\in L^\infty(0,+\infty;H^4), \chi_t\in L^\infty(0,+\infty;L^2)\cap L^2(0,+\infty);H^2),\\
&&\mu\in L^\infty(0,+\infty;H^2)\cap L^2(0,+\infty;H^4),\ \mu_t\in L^2(0,+\infty;L^2),\notag\\
 &&  \chi\in[-1,1],\ \ \mathrm{for \ all}\ (x,t)\in[0,L]\times[0,+\infty), \notag
\end{eqnarray}
and moreover
\begin{equation}\label{asymptotic stability for mixed boundary}
\lim_{t\rightarrow\infty}\sup_{x\in[0,L]}\big|(\rho-\bar\rho,u-\bar u,\chi-\bar\chi)\big|=0.
\end{equation}
\end{thm}

\begin{rmk} In above theorems, the solutions for two different boundary cases are proved to be asymptotically stable even for the large initial disturbance of the density and the large velocity data.
\end{rmk}

The key point of the proof for the main theorems is to construct some special energy functions and the approximate systems for \eqref{Euler-NSCH}  to get the upper and lower bound
 of the  density and the concentration difference.  The outline of this paper is organized as follows. In section 2, we   construct an approximate system for \eqref{Euler-NSCH}, and obtain the local existence of the solutions for  the approximate system. In section 3, we obtain the desired a priori estimates on strong solutions of the approximate system  and  take the limit for the approximate system, then give the proof of the main theorem 1.1. Moreover,  we generalized the Theorem 1.1 to the mixed boundary value problem \eqref{Euler-NSCH} and \eqref{mixed Boundary}, and  get the proof of theorem 1.2.

\section {\normalsize Local Existence and Uniqueness}
\setcounter{equation}{0}
In this section, we will give the local existence and uniqueness of the solution for the periodic boundary problem \eqref{Euler-NSCH}-\eqref{periodic boundary}. For this purpose,
 we should construct a family of smooth approximate functions for the non-smooth free energy density $f$ \eqref{the formula of free energy density} firstly. This process is strongly motivated by the work of Blowey-Elliott \cite{BE1991}.   The polynomial part of the free energy density   $\frac{1}{4}\big(\chi^2-1\big)^2$ is replaced by the twice continuously differentiable function $f_\lambda$  $(0<\lambda<1)$:
\begin{equation}\label{the approximate formula of free energy}
 f_\lambda(\chi)=\left\{\begin{array}{llll}
 \displaystyle\frac1{2\lambda}(\chi-(1+\frac\lambda2))^2+\frac{1}{4}\Big(\chi^2-1\Big)^2+\frac\lambda{24},
& \chi\geq1+\lambda, \vspace{2mm} \\
\displaystyle\frac{1}{4}\Big(\chi^2-1\Big)^2+\frac{1}{6\lambda^2}(\chi-1)^3,
& 1<\chi\leq1+\lambda, \vspace{2mm}\\
 \displaystyle\frac{1}{4}\Big(\chi^2-1\Big)^2, &  -1\leq\chi\leq1, \vspace{2mm} \\
\displaystyle\frac{1}{4}\Big(\chi^2-1\Big)^2-\frac{1}{6\lambda^2}(\chi+1)^3,
& -1-\lambda<\chi\leq-1, \vspace{2mm} \\
\displaystyle\frac1{2\lambda}(\chi+(1+\frac\lambda2))^2+\frac{1}{4}\Big(\chi^2-1\Big)^2+\frac\lambda{24},
& \chi\leq-1-\lambda.
 \end{array}\right.
\end{equation}
Directly  by simple calculation, one gets
\begin{equation}\label{the inequlity of f-lambda}
   f_\lambda(\chi)\geq\frac{1}{4}\Big(\chi^2-1\Big)^2\geq0,\  \forall\   0<\lambda<1.
\end{equation}
Moreover, one has
\begin{equation}\label{the derivative formula of f-lambda}
\frac{ \partial f_\lambda}{\partial\chi}(\chi)=
 \left\{\begin{array}{llll}
 \displaystyle\frac{1}{\lambda}\big(\chi-(1+\frac{\lambda}{2})\big)+\chi^3-\chi,\ & \chi\geq1+\lambda, \vspace{2mm}\\
\displaystyle\frac{1}{2\lambda^2}(\chi-1)^2+\chi^3-\chi,\ &1<\chi<1+\lambda, \vspace{2mm} \\
 \displaystyle \chi^3-\chi, &  -1\leq\chi\leq1, \vspace{2mm} \\
\displaystyle-\frac{1}{2\lambda^2}(\chi+1)^2+\chi^3-\chi,\ & -1-\lambda<\chi\leq-1, \vspace{2mm} \\
\displaystyle\frac{1}{\lambda}\big(\chi+(1+\frac{\lambda}{2})\big)+\chi^3-\chi,\ & \chi\leq-1-\lambda.
 \end{array}\right.
\end{equation}
\begin{equation}\label{the inequality-1 of f-lambda}
   \chi \frac{\partial f_\lambda}{\partial\chi}(\chi)\geq\chi^4-\chi^2,\  \chi^3 \frac{\partial f_\lambda}{\partial\chi}(\chi)\geq\chi^6-\chi^4,\ \forall\   0<\lambda<1.
\end{equation}
and
\begin{equation}\label{the second derivative formula of f-lambda}
\frac{ \partial^2 f_\lambda}{\partial\chi^2}(\chi)=
 \left\{\begin{array}{llll}
 \displaystyle\frac{1}{\lambda}+3\chi^2-1,\ & \chi\geq1+\lambda, \vspace{2mm}\\
\displaystyle\frac{1}{\lambda^2}(\chi-1)+3\chi^2-1,\ &1<\chi<1+\lambda, \vspace{2mm} \\
 \displaystyle 3\chi^2-1, &  -1\leq\chi\leq1, \vspace{2mm} \\
\displaystyle-\frac{1}{\lambda^2}(\chi+1)+3\chi^2-1,\ & -1-\lambda<\chi\leq-1, \vspace{2mm} \\
\displaystyle\frac{1}{\lambda}+3\chi^2-1,\ & \chi\leq-1-\lambda.
 \end{array}\right.
\end{equation}
We define the function $\beta_\lambda\in C^1(\mathbb{R})$ as follows
\begin{equation}\label{the approximate formula of beta-lambda}
 \beta_\lambda(\chi)=\lambda\Big(\chi-\chi^3+\frac{\partial f_\lambda}{\partial \chi}\Big)=
 \left\{\begin{array}{llll}
 \displaystyle\chi-(1+\frac{\lambda}{2}),\ & \chi\geq1+\lambda, \vspace{2mm}\\
\displaystyle\frac{1}{2\lambda}(\chi-1)^2,\ &1<\chi<1+\lambda, \vspace{2mm} \\
 \displaystyle 0, &  -1\leq\chi\leq1, \vspace{2mm} \\
\displaystyle-\frac{1}{2\lambda}(\chi+1)^2,\ & -1-\lambda<\chi\leq-1, \vspace{2mm} \\
\displaystyle\chi+(1+\frac{\lambda}{2}),\ & \chi\leq-1-\lambda.
 \end{array}\right.
\end{equation}

\begin{lem}
Suppose that $0<\lambda<1$, then $\beta_\lambda$ satisfies the following

1) $\beta_\lambda$ is a Lipschitz continuous function and $0\leq\beta_\lambda\leq1$.

2)\begin{equation}\label{the limit of beta-lambda}
\beta(\chi):=\lim_{\lambda\rightarrow0}\beta_\lambda(\chi)=
\left\{\begin{array}{llll}
 \displaystyle\chi-1,& \chi>1,\\
 \displaystyle 0, &  -1\leq\chi\leq1,\\
\displaystyle\chi+1,& \chi<-1.
 \end{array}\right.
 \end{equation}

3) $\beta(\chi)$ is a Lipschitz continuous function, and
$$|\beta(\chi)-\beta_{\lambda}(\chi)|\leq\frac{1}{2}\lambda, \ \  |\beta(\chi_1)-\beta(\chi_2)|\leq|\chi_1-\chi_2|.$$
\end{lem}

Making use of the smooth polynomial approximate free energy density $f_\lambda$ \eqref{the approximate formula of free energy}, we construct an approximate problem for the \eqref{Euler-NSCH}-\eqref{periodic boundary} as follows:
\begin{equation}\label{penalized-NSCH}
\left\{\begin{array}{llll}
\displaystyle \rho_t+(\rho u)_x=0,  \\
\displaystyle \rho u_t+\rho uu_x+p_x=\nu u_{xx}-\frac{\epsilon}{2}\big(\chi_x^2\big)_x,\\
\displaystyle\rho\chi_t+\rho u\chi_x=\mu_{xx},\\
\displaystyle  \rho\mu=\frac{1}{\epsilon}\rho\frac{\partial f_\lambda(\rho,\chi)}{\partial \chi}-\epsilon\chi_{xx},\\
\displaystyle(\rho,u,\chi)(x,t)=(\rho,u,\chi)(x+L,t),\\
\displaystyle(\rho,u,\chi)(x,0)=(\rho_0,u_0,\chi_0).
\end{array}\right.
\end{equation}
For $\forall m>0$, $M>0$, $B>0$, $I\subseteq\mathbb{R}$, we define
\begin{eqnarray}\label{function space}
 &&X_{m,M,B}(I)\equiv\Big\{(\rho,u,\chi)\Big|(\rho,u)\in C^0(I,H_{\mathrm{per}}^2),\chi\in C^0(I;H_{\mathrm{per}}^4),\notag\\
 &&\qquad\qquad\qquad\qquad\qquad\rho\in L^2(I;H_{\mathrm{per}}^2), u\in L^2(I;H_{\mathrm{per}}^3),\chi\in L^2(I;H_{\mathrm{per}}^5),\notag\\
 &&\qquad\qquad\qquad\qquad\inf_{t\in I, x\in\mathbb{R}}\rho(x,t)\geq m,\sup_{t\in I}\|(\rho,u)\|^2_2\leq M,\sup_{t\in I}\|\chi-\bar\chi\|_4^2\leq B
 \Big\}.
\end{eqnarray}
In particular, for fixed $B>0$, we define
\begin{equation}\label{function space-2}
 X_{0,+\infty,B}(I)\equiv\bigcup_{m>0,M>0}X_{m,M,B}(I).
\end{equation}

\begin{prop}\label{local existence and uniqueness for approximate solution}
Assume that $(\rho_0,u_0,\chi_0)$ satisfies   the periodic boundary \eqref{periodic boundary} with the regularities given in \eqref{initial data of density}-\eqref{Compatibility condition chi}. For $\forall m>0$, $M>0$ and $B>0$, if $\inf_{x\in\mathbb{R}}\rho_0(x)\geq m$, $\|(\rho_0,u_0)\|^2_2\leq M$, $\|\chi_0\|_4^2\leq B$,
then there exist a small time $T_*>0$ and a unique strong solution $(\rho_\lambda,u_\lambda,\chi_\lambda)$ to the  approximate problem \eqref{penalized-NSCH} with the  smooth approximate function $f_\lambda$ \eqref{the approximate formula of free energy}, such that $(\rho_\lambda,u_\lambda,\chi_\lambda)\in X_{\frac{m}{2},2M,2B}([0,T_*])$.
\end{prop}
\begin{proof} Because of the difficulty with the coefficients of \eqref{penalized-NSCH}$_{3,4}$ when using the result of the linear parabolic equation, we need to approximate the initial conditions. Constructing a approximate function sequence $(\rho_0^\delta,u^\delta_0,\chi^\delta_0)$ for the initial data,
$(\rho_0^\delta,u^\delta_0,\chi^\delta_0)\in H_{\mathrm{per}}^3$,  $(\rho_0^\delta,u^\delta_0,\chi^\delta_0)$ satisfies the periodic boundary \eqref{periodic boundary} with the regularities given in \eqref{initial data}-\eqref{Compatibility condition chi}, and
\begin{equation}\label{approximate sequence for initial data}
\lim_{\delta\rightarrow 0}\big(\|\rho_0^\delta-\rho_0\|_2+\|u_0^\delta-u_0\|_2+\|\chi_0^\delta-\chi_0\|_4\big)=0.
\end{equation}
 Taking $0<T_0<+\infty$, for $\forall m>0$, $M>0$ and $B>0$, we define
\begin{eqnarray}\label{smooth function space}
 &&\tilde{X}_{m,M,B}([0,T_0])\equiv\Big\{(\rho,u,\chi)\Big|(\rho,u)\in C^0([0,T_0],H_{\mathrm{per}}^3),\chi\in C^0([0,T_0];H_{\mathrm{per}}^4),\notag\\
 &&\qquad\qquad\qquad\rho\in L^2([0,T_0];H_{\mathrm{per}}^3), u\in L^2([0,T_0];H_{\mathrm{per}}^4),\chi\in L^2([0,T_0];H_{\mathrm{per}}^5),\notag\\
 &&\qquad\qquad\inf_{t\in [0,T_0], x\in\mathbb{R}}\rho(x,t)\geq m,\sup_{t\in [0,T_0]}\|(\rho,u)\|^2_3\leq M,\sup_{t\in [0,T_0]}\|\chi-\bar\chi\|_4^2\leq B
 \Big\}.
\end{eqnarray}
Constructing an  iterative sequence $(\rho^{(n)},u^{(n)},\chi^{(n)})$,$n=0,1,2\cdots$, satisfies $(\rho^{(0)},u^{(0)},\chi^{(0)})=(\rho_0^\delta,u^\delta_0,\chi^\delta_0)$,
and the  following iterative scheme
\begin{equation}\label{iterative-NSCH}
\left\{\begin{array}{llll}
\displaystyle \rho^{(n)}_t+(\rho^{(n)} u^{(n-1)})_x=0,  \\
\displaystyle \rho^{(n)} u^{(n)}_t+\rho^{(n)} u^{(n-1)}u^{(n)}_x+p^{(n)}_x=\nu u^{(n)}_{xx}-\frac{\epsilon}{2}\big((\chi_x^{(n)})^2\big)_x,\\
\displaystyle\rho^{(n)}\chi^{(n)}_t+\rho^{(n)} u^{(n-1)}\chi^{(n)}_x=\mu^{(n)}_{xx},\\
\displaystyle  \rho^{(n)}\mu^{(n)}=\frac{1}{\epsilon}\rho^{(n)}\frac{\partial f_\lambda}{\partial\chi}(\chi^{(n-1)})-\epsilon\chi^{(n)}_{xx},\\
\displaystyle(\rho^{(n)},u^{(n)},\chi^{(n)})(x,t)=(\rho^{(n)},u^{(n)},\chi^{(n)})(x+L,t),\\
(\rho^{(n)},u^{(n)},\chi^{(n)})\Big|_{t=0}=\big(\rho_0,u_0,\chi_0\big)(x),
\end{array}\right.
\end{equation}
here $p^{(n)}=a\big(\rho^{(n)}\big)^\gamma$.  Suppose that $(\rho^{(n-1)},u^{(n-1)},\chi^{(n-1)})$ $\in$ $\tilde{X}_{\frac{m}{2},2M,2B}([0,T_0])$,
the existence and uniqueness of the strong solution $\rho^{(n)}$ for the  equation \eqref{iterative-NSCH}$_1$ can be obtained by the basic linear hyperbolic theory, we refer to \cite{S94} and the references therein. In addition, $\rho^{(n)}$  satisfies
\begin{equation}\label{estimate for conservation of mass}
 \sup_{[0,T_0]}\Big(\|\rho^{(n)}\|_{H^3_{\mathrm{per}}}+\|\rho^{(n)}_t\|_{H^2_{\mathrm{per}}}+\|(\rho^{(n)})^{-1}\|_{L^\infty_{\mathrm{per}}}\Big)\leq C(m,M,T_0).
\end{equation}
Further, rewriting \eqref{iterative-NSCH}$_{3,4}$ as
\begin{eqnarray}\label{conservation of concentration difference}
\chi^{(n)}_{t}&=&-u^{(n-1)}\chi^{(n)}_{x}-\frac{\epsilon\chi^{(n)}_{xxxx}}{(\rho^{(n)})^2}+\frac{2\epsilon\rho^{(n)}_{x}}{(\rho^{(n)})^3}\chi^{(n)}_{xxx}
+\epsilon\Big(\frac{\rho^{(n)}_{xx}}{(\rho^{(n)})^3}-\frac{2(\rho^{(n)}_{x})^2}{(\rho^{(n)})^4}\Big)\chi^{(n)}_{xx}\notag  \\
&&+\frac{\frac{\partial^2 f_\lambda}{\partial\chi^2}(\chi^{(n-1)})\chi^{(n-1)}_{xx}}{\epsilon\rho^{(n)}}
+\frac{\frac{\partial^3 f_\lambda}{\partial\chi^3}(\chi^{(n-1)})(\chi^{(n-1)}_{x})^2}{\epsilon\rho^{(n)}},
\end{eqnarray}
following from \cite{CM1995},\cite{LWZ2007} and the references therein,   the existence and uniqueness of the solution $\chi^{(n)}$ for \eqref{conservation of concentration difference} with the initial boundary value \eqref{periodic boundary} can be obtained, and $\chi^{(n)}$ satisfies
\begin{equation}\label{estimate for conservation of concentration difference}
 \sup_{[0,T_0]}\Big(\|\chi^{(n)}\|_{H^4_{\mathrm{per}}}+\|\chi^{(n)}_t\|_{L^2_{\mathrm{per}}}+\|\mu^{(n)}\|_{H^2_{\mathrm{per}}}\Big)\leq C(m,M,T_0).
\end{equation}
Moreover, rewriting the \eqref{iterative-NSCH}$_{2}$ as
\begin{equation}\label{Conservation of momentum}
u^{(n)}_{t}=\frac{\nu}{\rho^{(n)}}u^{(n)}_{xx}-u^{(n-1)}u^{(n)}_{x}-\frac{p'_\rho(\rho^{(n)})}{\rho^{(n)}}\rho^{(n)}_{x}-\frac{\epsilon}{\rho^{(n)}_{x}}\chi^{(n)}_{x}\chi^{(n)}_{xx},
\end{equation}
similar as above, the existence and uniqueness of the solution $u^{(n)}$ for \eqref{Conservation of momentum} with the initial boundary value \eqref{periodic boundary} can be obtained, and $u^{(n)}$ satisfies
\begin{equation}\label{estimate for conservation of momentum}
 \sup_{[0,T_0]}\Big(\|u^{(n)}\|_{H^3_{\mathrm{per}}}+\|u^{(n)}_t\|_{H^1_{\mathrm{per}}}\Big)\leq C(m,M,T_0).
\end{equation}
Now, all we have to do is to  show that, there exists a $T_*>0$ small enough,
$(\rho^{(n)},u^{(n)},\chi^{(n)})\in \tilde{X}_{\frac{m}{2},2M,2B}([0,T_*])$. We use mathematical induction and the energy estimate method to prove this assertion.
By using \eqref{iterative-NSCH}$_1$, applying the method of characteristics,  $\rho^{(n)}$ can be expressed by
\begin{equation}\label{expression for density}
\displaystyle \rho^{(n)}(x,t)=\rho_0(X(x,t;0))e^{-\int_0^t u^{(n-1)}_x(X(x,t;s),s)ds},
\end{equation}
for $(x,t)\in\mathbb{R}\times[0,T_0]$,
where $X\in C(\mathbb{R}\times[0,T_0]\times[0,T_0])$ is the solution to the initial value problem
\begin{equation}\label{lem1-2}
\left\{\begin{array}{llll}
\displaystyle \frac{d}{d t}X(x,t;s)=u^{(n-1)}(X(x,t;s),t),&~ 0\leq t\leq T,  \\
\displaystyle X(x,s;s)=x,&~ 0\leq s\leq T,~x\in\mathbb{R}.
\end{array}\right.
\end{equation}
The regularity of $X(x,t;s)$ gives $\rho^{(n)}\in C^0([0,T_0];H^3)$. Multiplying  \eqref{iterative-NSCH}$_1$ by $\rho^{(n)}$ and integrating it by parts, one has
\begin{eqnarray}\label{L2-energy for rho-n}
  &&\frac{1}{2}\frac{d}{dt}\|\rho^{(n)}\|^2\leq \frac{1}{2}\int_0^L u^{(n-1)}_x(\rho^{(n)})^2dx+\bar{\rho}\int_0^L |u^{(n-1)}_x||\rho^{(n)}|dx\notag\\
  &&\leq\|u^{(n-1)}_x\|^{\frac12}\|u^{(n-1)}_{xx}\|^{\frac12}\|\rho^{(n)}\|^2+\frac{\bar{\rho}}{2}\big(\|u^{(n-1)}_x\|^2+\|\rho^{(n)}\|^2\big).
  \end{eqnarray}
From  Gronwall's inequality,  one obtains
\begin{eqnarray}\label{the estimate of rho}
 \|\rho^{(n)}(t)\|^2
 \leq\big(\|\rho_0\|^2+\bar\rho\int_0^{T_0}\|u^{(n-1)}_x(\tau)\|^2d\tau\big)e^{\big[\int_0^{T_0}\big(2\|u^{(n-1)}_x(\tau)^\frac{1}{2}\|u^{(n-1)}_{xx}(\tau)\|^\frac{1}{2}+1\big)d\tau\big]}.
 \end{eqnarray}
Differentiating both side of \eqref{iterative-NSCH}$_1$ with respect to $x$, multiplying it by $\rho^{(n)}_x$ and integrating by parts, one has
\begin{eqnarray}\label{L2-energy for derivative of rho-n}
\frac{1}{2}\frac{d}{dt}\|\rho^{(n)}_x\|^2&\leq& 2\int_0^L u^{(n-1)}_x(\rho^{(n)}_x)^2dx+\int_0^L|\rho^{(n)}-\bar\rho||\rho^{(n)}_x||u^{(n-1)}_{xx}|dx+\bar{\rho}\int_0^L |u^{(n-1)}_{xx}||\rho_x^{(n)}|dx\notag\\
  &\leq&4\|u^{(n-1)}_x\|^{\frac12}\|u^{(n-1)}_{xx}\|^{\frac12}\|\rho_x^{(n)}\|^2+2\|u^{(n-1)}_{xx}\|\|\rho^{(n)}-\bar\rho\|^{\frac12}\|\rho^{(n)}_x\|^{\frac32}\notag\\
  &&+\bar{\rho}\|u^{(n-1)}_{xx}\|\|\rho^{(n)}_{xx}\|.
  \end{eqnarray}
Similarly, for the second derivatives and third derivatives, one gets
\begin{eqnarray}\label{L2-energy for the second derivative of rho-n}
\frac{1}{2}\frac{d}{dt}\|\rho^{(n)}_{xx}\|^2&\leq& C\Big(\int_0^L |u^{(n-1)}_x|(\rho^{(n)}_{xx})^2dx+\int_0^L|\rho^{(n)}
-\bar\rho||\rho^{(n)}_{xx}||u^{(n-1)}_{xxx}|dx\notag\\
&&+\int_0^L|\rho_x^{(n)}||\rho^{(n)}_{xx}||u^{(n-1)}_{xx}|dx\Big)+\bar{\rho}\int_0^L |u^{(n-1)}_{xxx}||\rho_{xx}^{(n)}|dx\notag\\
  &\leq&C\Big(\|u^{(n-1)}_x\|^{\frac12}\|u^{(n-1)}_{xx}\|^{\frac12}\|\rho_{xx}^{(n)}\|^2+
  \|u^{(n-1)}_{xxx}\|\|\rho^{(n)}-\bar\rho\|^{\frac12}\|\rho^{(n)}_x\|^{\frac12}\|\rho^{(n)}_{xx}\|\notag\\
  &&+ \|u^{(n-1)}_{xx}\|^{\frac12} \|u^{(n-1)}_{xxx}\|^{\frac12}\|\rho^{(n)}_x\|\|\rho^{(n)}_{xx}\|\Big)+\bar{\rho}\|u^{(n-1)}_{xxx}\|\|\rho^{(n)}_{xx}\|,
  \end{eqnarray}
and
\begin{eqnarray}\label{L2-energy for the third derivative of rho-n}
&&\frac{1}{2}\frac{d}{dt}\|\rho^{(n)}_{xxx}\|^2\leq C\Big(\int_0^L u^{(n-1)}_x(\rho^{(n)}_{xxx})^2dx+
\int_0^L|\rho^{(n)}-\bar\rho||\rho^{(n)}_{xxx}||u^{(n-1)}_{xxxx}|dx\notag\\
&&+\int_0^L|\rho_x^{(n)}||\rho^{(n)}_{xxx}||u^{(n-1)}_{xxx}|dx
+\int_0^L|u^{(n-1)}_{xx}||\rho^{(n)}_{xx}||\rho^{(n)}_{xxx}|dx\Big)+\bar{\rho}\int_0^L |u^{(n-1)}_{xxxx}||\rho_{xxx}^{(n)}|dx\notag\\
  &&\leq C\Big(\|u^{(n-1)}_x\|^{\frac12}\|u^{(n-1)}_{xx}\|^{\frac12}\|\rho_{xxx}^{(n)}\|^2+
  \|u^{(n-1)}_{xxxx}\|\|\rho^{(n)}-\bar\rho\|^{\frac12}\|\rho^{(n)}_x\|^{\frac12}\|\rho^{(n)}_{xxx}\|\notag\\
  &&+\|\rho^{(n)}_x\|^{\frac12}\|\rho^{(n)}_{xx}\|^{\frac12}\|u^{(n-1)}_{xxx}\|\|\rho^{(n)}_{xxx}\|+ \|u^{(n-1)}_{xx}\|^{\frac12} \|u^{(n-1)}_{xxx}\|^{\frac12}\|\rho^{(n)}_{xx}\|\|\rho^{(n)}_{xxx}\|\Big)\notag\\
  &&+\bar{\rho}\|u^{(n-1)}_{xxxx}\|\|\rho^{(n)}_{xxx}\|.
  \end{eqnarray}
Hence,  adding \eqref{the estimate of rho}--\eqref{L2-energy for the third derivative of rho-n}, from Gronwall's inequality, one has
\begin{equation}\label{energy inequality for the second derivative approximation of rho}
  \|\rho^{(n)}(t)\|_2\leq C\big(M+C(M)T_0\big)e^{C(M)T_0},
\end{equation}
\begin{equation}\label{energy inequality for the third derivative approximation of rho}
  \|\rho_{xxx}^{(n)}(t)\|\leq C\big(M+C(M)\sqrt{T_0}\|u_{xxxx}^{(n-1)}\|_{L^2(0,T_0;L^2(\mathbb{R}))}\big)e^{C(M)T_0},
\end{equation}
and in the same way, one obtains
\begin{equation}\label{energy inequality for t-derivative approximation of rho}
  \|\rho^{(n)}_t\|_1\leq C\Big(\big(M+C(M)T_0M\big)e^{C(M)T_0}+1\Big),
\end{equation}
\begin{equation}\label{energy inequality for t-derivative approximation of rho-1}
  \|\rho^{(n)}_t\|_2\leq C\Big(\big(M+\sqrt{T_0}C(M)\|u_{xxxx}^{(n-1)}\|_{L^2(0,T_0;L^2(\mathbb{R}))}\big)e^{C(M)T_0}+1\Big).
\end{equation}
From \eqref{expression for density} and \eqref{energy inequality for the second derivative approximation of rho}, for $T_0$ small enough, which be called as $T_1$, $\rho^{(n)}$ satisfies
\begin{equation}\label{the lower and upper bound of density}
 \inf_{t\in[0,T_1],x\in\mathbb{R}}\rho^{(n)}\geq\frac{m}{2},\ \ \sup_{t\in[0,T_1],x\in\mathbb{R}}\|\rho^{(n)}-\bar\rho\|_3^2\leq2M.
\end{equation}
Multiplying \eqref{iterative-NSCH}$_3$ by $\chi^{(n)}$,   integrating over $[0,L]$ by parts, one has
\begin{eqnarray}\label{energy estimate-1 for approximation chi }
 &&\frac{1}{2}\frac{d}{dt}\int_0^L\rho^{(n)}\big(\chi^{(n)}\big)^{2}dx=\int_0^L\mu^{(n)}\chi^{(n)}_{xx}dx\notag\\
 &&=\frac{1}{\epsilon}\int_0^L\Big(\frac{\partial f_\lambda}{\partial\chi}(\chi^{(n-1)})\Big)\chi^{(n)}_{xx}dx-\epsilon\int_0^L\frac{1}{\rho^{(n)}}(\chi^{(n)}_{xx})^2dx
\end{eqnarray}
thus
\begin{eqnarray}\label{energy inequality-1 for approximation chi}
 \sup_{[0,T_1]}\|(\chi^{(n)}-\bar\chi)(t)\|^{2}+\int_0^{T_1}\|\chi^{(n)}_{xx}(\tau)\|^2d\tau\leq B+C(\epsilon, m,M,B)T_1.
 \end{eqnarray}
Multiplying  \eqref{iterative-NSCH}$_3$  by $\mu^{(n)}$,  integrating over $[0,L]$ by parts, one gets
\begin{eqnarray}\label{energy inequality-2 for approximation chi}
&&\frac{\epsilon}{2}\frac{d}{dt}\int_0^L\big|\chi^{(n)}_x\big|^2dx+\int_0^L|\mu_x^{(n)}|^2dx\notag\\
&&=\frac{1}{\epsilon}\int_0^L\Big(\frac{\partial^2 f_\lambda}{\partial\chi^2}(\rho^{(n-1)},\chi^{(n-1)})\Big)\chi^{(n-1)}_x\mu^{(n)}_xdx
-\frac{\epsilon}{2}\int_0^L u^{(n-1)}_x\big(\chi^{(n)}_x\big)^2dx,\notag\\
&&\leq\frac{1}{2}\int_0^L(\mu_x^{(n)})^2dx+\frac{1}{2\epsilon}\int_0^L\Big(\frac{\partial^2
f_\lambda}{\partial\chi^2}(\rho^{(n-1)},\chi^{(n-1)})\Big)^2(\chi^{(n-1)}_x)^2dx\notag\\
&&\qquad+
\frac{\epsilon}{2}\|u^{(n-1)}_x\|^{\frac12}\|u^{(n-1)}_{xx}\|^{\frac12}\int_0^L \big(\chi^{(n)}_x\big)^2dx,\notag
\end{eqnarray}
therefore
\begin{eqnarray}\label{energy estimate-2 for approximation chi}
\frac{\epsilon}{2}\frac{d}{dt}\int_0^L\big|\chi^{(n)}_x\big|^2dx+\frac{1}{2}\int_0^L|\mu_x^{(n)}|^2dx\leq\frac{3B^2(B^2-1)^2}{2\epsilon}+
\frac{\epsilon B}{2}\int_0^L \big(\chi^{(n)}_x\big)^2dx,
\end{eqnarray}
following from Gronwall's inequality, one has
\begin{equation}\label{energy estimate-2-1 for approximation chi}
\sup_{0\leq t\leq T_1}\epsilon\|\chi^{(n)}_x\|^2\leq\Big(\delta_0+\frac{3B^2(B^2-1)^2}{\epsilon}T_1\Big)e^{\epsilon B T_1},
\end{equation}
substituting the  inequality \eqref{energy estimate-2-1 for approximation chi} into \eqref{energy estimate-2 for approximation chi}, one obtains
\begin{eqnarray}\label{energy estimate-2-2 for approximation chi}
\int_0^{T_1}\|\mu_x^{(n)}\|^2dx\leq\Big(B+\frac{3B^2(B^2-1)^2}{\epsilon}T_1\Big)+\frac{B}{2}\Big(B+\frac{3B^2(B^2-1)^2}{\epsilon}T_1\Big)e^{\epsilon B T_1}T_1,
\end{eqnarray}
combining the inequalities \eqref{energy estimate-2-1 for approximation chi},\eqref{energy estimate-2-2 for approximation chi} and \eqref{iterative-NSCH}$_4$, one has
\begin{equation}\label{energy estimate-2-3 for approximation chi}
\int_0^{T_0}\int_0^L|\chi^{(n)}_{xxx}|^2dx\leq C\big(B+C(\epsilon,m,M,B)T_1\big).
\end{equation}
Differentiating \eqref{iterative-NSCH}$_3$ with respect to $t$, one gets
\begin{equation}\label{differentiate  approximation for Cahn-Hilliard with respect of t}
\rho^{(n)}\chi^{(n)}_{tt}+\rho^{(n)}_t\chi^{(n)}_t+\rho^{(n)}_t u^{(n\!-\!1)}\chi^{(n)}_x+\rho^{(n)} u^{(n\!-\!1)}_t\chi_x^{(n)}+\rho^{(n)}u^{(n\!-\!1)}\chi^{(n)}_{xt}=\mu^{(n)}_{xxt},
 \end{equation}
multiplying \eqref{differentiate  approximation for Cahn-Hilliard with respect of t} by $\chi^{(n)}_t$, integrating over $[0,L]$, one has
\begin{eqnarray}\label{energy estimate-3 for  t derivative approximation chi}
&&\frac{1}{2}\int_0^L\rho^{(n)}|\chi^{(n)}_t|^2dx+\int_0^L\frac{\epsilon}{\rho^{(n)}}|\chi^{(n)}_{xxt}|^2dx\notag\\
&&=-\int_0^L\Big(\rho_t^{(n)}|\chi_t^{(n)}|^2+\rho_t^{(n)}u^{(n-1)}\chi^{(n)}_x\chi^{(n)}_t+\rho^{(n)}u_t^{(n-1)}\chi_x^{(n)}\chi_t^{(n)}+\big(\frac{1}{\rho^{(n)}}\big)_t\chi^{(n)}_{xx}
\chi^{(n)}_{xxt}\Big)dx\notag\\
&&\qquad+\frac{1}{\epsilon}\int_0^L\Big(\frac{\partial^2 f_\lambda}{\partial\chi^2}(\rho^{(n-1)},\chi^{(n-1)})\Big)\chi_t^{(n-1)}\chi^{(n)}_{xxt}dx\notag\\
&&\leq \int_0^L\frac{\epsilon}{2\rho^{(n)}}|\chi^{(n)}_{xxt}|^2dx+C(M)\int_0^L\rho^{(n)}|\chi^{(n)}_t|^2+C(\epsilon,M,B)(1+T_1).
\end{eqnarray}
 Noting that from \eqref{conservation of concentration difference},
we see that
\begin{equation}\label{intial value for chi-1}
 \|\sqrt{\rho^{(n)}}\chi^{(n)}_t(0)\|\leq C(\|\rho_0\|_2,\|u_0\|,\|\chi_0\|_4),
\end{equation}
combining with \eqref{energy estimate-3 for  t derivative approximation chi}, and from Gronwall's inequality, one obtains
\begin{eqnarray}\label{energy inequality-3 for t detrvative approximation chi}
&&\sup_{[0,T_1]}\|\chi^{(n)}_t(\tau)\|^2+\int_0^{T_1}\big(\|\chi_{xxt}^{(n)}(\tau)\|^2+\|\mu_t^{(n)}(\tau)\|^2+\|\mu_{xxxx}^{(n)}(\tau)\|^2\big)d\tau\notag\\
&&\ \ \leq C\big(B+C(\epsilon,m,M,B)T_1\big).
\end{eqnarray}
Moreover, in the same way, one gets
\begin{equation}\label{energy inequality-4 for the second derivative of approxiamation chi}
  \sup_{[0,T_1]}\|\chi^{(n)}_{xx}(\tau)\|^2+\int_0^{T_1}\|\chi^{(n)}_{xxxx}(\tau)\|^2d\tau\leq C(B+C(\epsilon,m,M,B)T_1).
\end{equation}
and
\begin{eqnarray}\label{energy inequality-5 for the forth derivative of approxiamation chi}
 && \sup_{[0,T_1]}\big(\|\mu^{(n)}(t)\|^2+\|\mu^{(n)}_{xx}(t)\|^2+\|\chi^{(n)}_{xxx}(t)\|^2+ \|\chi^{(n)}_{xxxx}(t)\|^2\big)+\int_0^{T_1}\|\chi^{(n)}_{xxxxx}(\tau)\|^2d\tau\notag\\
  &&\quad\leq C(B+C(\epsilon,m,M,B)T_1).
\end{eqnarray}
Multiplying \eqref{iterative-NSCH}$_2$ by $u^{(n)}$, integrating over $[0,L]$, one obtains
\begin{eqnarray}\label{energy estimate-1-1 for approximation velocity}
\frac{d}{dt}\int_0^L\frac{1}{2}\rho^{(n)}|u^{(n)}|^2
dx+\int_0^L p^{(n)}_xu^{(n)}dx+\nu\int_0^L|u_x^{(n)}|^2dx=-\frac{1}{2}\int_0^L\big(\chi_x^{(n)}\big)^2_xu^{(n)}dx.
\end{eqnarray}
We define
\begin{equation}\label{the function G}
  G(\rho)=\rho\int_{\bar\rho}^\rho s^{-2}(p(s)-\bar{p})ds,
\end{equation}
where $\bar p=p(\bar\rho)$, so that by \eqref{iterative-NSCH}$_1$, one gets
\begin{equation}\label{the mass equation}
  G(\rho^{(n)})_t+\big(G(\rho^{(n)})u^{(n-1)}\big)_x+\big(p(\rho^{(n)})-\bar p\big)u^{(n-1)}_x=0.
\end{equation}
Integrating the result and adding it to \eqref{energy estimate-1-1 for approximation velocity}, then one has
\begin{equation}\label{energy estimate-1-2 for approximation velocity}
 \frac{d}{dt}\int_0^L\big(\frac{1}{2}\rho^{(n)}|u^{(n)}|^2+G(\rho^{(n)})\big)dx
+\nu\int_0^L|u_x^{(n)}|^2dx=-\frac{1}{2}\int_0^L\big(\chi_x^{(n)}\big)^2_xu^{(n)}dx.
\end{equation}
thus, one obtains
\begin{equation}\label{energy estimate-1 for approximation velocity}
 \sup_{[0,T_1]}\big(\|u^{(n)}(t)\|^2+\|\rho^{(n)}(t)-\bar\rho\|^2\big)
+\int_0^{T_1}\|u_x^{(n)}(\tau)\|^2d\tau\leq C(M+C(\epsilon,m,M,B)T_1).
\end{equation}
Differentiating both side of \eqref{iterative-NSCH}$_2$ with respect to $x$, multiplying it by $u^{(n)}_t$, integrating over $[0,L]$ by parts, one has
\begin{eqnarray}\label{energy estimate-2-1 for approximation velocity}
&&\frac{1}{2}\frac{d}{dt}\|u_x^{(n)}\|^2+\int_0^L\rho^{(n)}|u^{(n)}_t|^2dx\notag\\
&&=-\int_0^L\rho^{(n)}u^{(n-1)}u^{(n)}_xu^{(n)}_tdx-\int_0^L p'(\rho^{(n)})\rho^{(n)}_xu^{(n)}_tdx-\epsilon\int_0^L\chi^{(n)}_x\chi^{(n)}_{xx}u^{(n)}_tdx\notag\\
&&\leq\frac{1}{2}\int_0^L\rho^{(n)}|u^{(n)}_t|^2dx+c\|\rho^{(n)}\|^{\frac12}\|\rho^{(n)}_x\|^{\frac12}\|u^{(n-1)}\|\|u^{(n-1)}_x\|\int_0^L |u^{(n)}_x|^2dx+C(m,M,B),\notag
\end{eqnarray}
then from Gronwall's inequality, one gets
\begin{equation}\label{energy estimate-2 for approximation velocity}
\sup_{[0,T_1]}\|u_x^{(n)}(t)\|^2
+\int_0^{T_1}\int_0^L\rho^{(n)}|u_t^{(n)}|^2dxd\tau\leq C(M+C(\epsilon,m,M,B)T_1).
\end{equation}
In the same way, for the higher derivative of $u^{(n)}$, one obtains
\begin{eqnarray}\label{energy estimate-3 for approximation velocity}
&&\sup_{[0,T_1]}\big(\|u_t^{(n)}(t)\|^2+\|u_{xx}^{(n)}(t)\|^2+\|u_{xt}^{(n)}(t)\|^2+\|u_{xxx}^{(n)}(t)\|^2\big)
\notag\\
&&+\int_0^{T_1}\int_0^L\big(|u_{xt}^{(n)}|^2+|u_{xxt}^{(n)}|^2+|u_{xxx}^{(n)}|^2\big)dxd\tau\leq C(M+C(\epsilon,m,M,B)T_1).
\end{eqnarray}
From the  energy inequalities \eqref{energy inequality for the second derivative approximation of rho}-\eqref{the lower and upper bound of density},
\eqref{energy inequality-1 for approximation chi}-\eqref{energy estimate-2-3 for approximation chi},
\eqref{energy inequality-3 for t detrvative approximation chi}-\eqref{energy inequality-5 for the forth derivative of approxiamation chi},
\eqref{energy estimate-1 for approximation velocity}-\eqref{energy estimate-3 for approximation velocity},
 we can choose $T$ small enough, without loss of generality
, say $T_*>0$, and all
the  conditions stated in the definition of $\tilde{X}_{\frac{m}{2},2M,2B}([0,T_*])$  are guaranteed. Thereby
 the iterative sequence $(\rho^{(n)},u^{(n)},\chi^{(n)})\in \tilde{X}_{\frac{m}{2},2M,2B}([0,T_*])$. Moreover, $(\rho^{(n)},u^{(n)},\chi^{(n)})$ is the Cauchy sequence in $C^0([0,T_*];H_{\mathrm{per}}^1)\times$ $C^0([0,T_*];H_{\mathrm{per}}^1)\times$ $C^0([0,T_*];H_{\mathrm{per}}^1)$. We define $(\rho_\lambda,u_\lambda,\chi_\lambda)$ the limit of the sequence $(\rho^{(n)},u^{(n)},\chi^{(n)})$. It is easy to know that $(\rho_\lambda,u_\lambda,\chi_\lambda)$ is the uniqueness solution of the systems \eqref{penalized-NSCH} in the space $\tilde{X}_{\frac{m}{2},2M,2B}([0,T_*])$.
Finially, letting $\delta\rightarrow 0$, we can easily derive the limit $(v,u,\chi)$ of sequence $(\rho^\delta,u^\delta,\chi^\delta)$ satisfies,  $(v,u,\chi)\in X_{m,M,B}([0,T_0])$, and it is  a solution of system \eqref{penalized-NSCH} with the initial data $(\rho_0,u_0,\chi_0)$.  The proof of Proposition \ref{local existence and uniqueness for approximate solution} is completed.
\end{proof}

\section {A Priori Estimates}
\setcounter{equation}{0}
In this section, we will present the desired estimates and global existence of the solution
for the approximate periodic boundary problem \eqref{penalized-NSCH}, then we will take the limit of the approximate equation \eqref{penalized-NSCH}, and
give the proof of the main theorem 1.1. Furthermore, we will extend these results to mixed boundary problem \eqref{Euler-NSCH},\eqref{mixed Boundary} and obtain the Theorem 1.2. Based on the
local existence and the a priori estimates, we may obtain the global solution by the continuity extension argument developed in the previous papers \cite{HM2009}-\cite{HMX2006}, \cite{MN1985}-\cite{SYZ2017} and the references therein.

\begin{prop}\label{a priori estimate proposition}
Let $\bar{\chi}\in A_{\textrm{stable}}\cap[-1,1]$ be fixed. Assume that $(\rho_0,u_0,\chi_0)$ satisfies  \eqref{initial data of density}-\eqref{Compatibility condition chi}. Then there exist  positive constants $C$ and $\varepsilon_1$ depending only on $\nu$, $\epsilon$,  such that if
\begin{equation}\label{small condition for chi}
\|(\chi_0-\bar\chi)\|_4^2\leq \varepsilon_1,
\end{equation}
then the  periodic boundary problem \eqref{Euler-NSCH},\eqref{periodic boundary} has a solution $(\rho,u,\chi)\in X_{0,+\infty,\varepsilon_1}([0,T])$,
and
\begin{eqnarray}\label{a priori estimate}
\left.\begin{array}{llll}
 \displaystyle\sup_{t\in [0,T]}\big(\|(\rho,u)\|^2_2+\|\chi(t)\|_4^2+\|\rho_t\|^2_2+\|u_t\|^2_1+\|\chi_t\|^2\big) \\
\displaystyle \qquad+\int_0^T\big(\|\rho_x(t)\|_1^2+\|u_x(t)\|_2^2+\|\chi_x(t)\|_4^2+\|u_t(t)\|_2^2+\|\chi_t(t)\|_2^2\big) dt\\
\displaystyle\leq C\Big(\|(\rho_0,u_0)\|^2_2+\|\chi_0\|_4^2\Big).
\end{array}\right.
 \end{eqnarray}
\end{prop}
\begin{proof}
For fixed positive constant  $\bar{\chi}\in A_{\textrm{stable}}$,  from Sobolev theorem, one can choose $B_0$ small enough, such that for $0<B\leq B_0$,
\begin{equation}\label{bound of density and chi}
  \chi\in A_{\textrm{stable}}.
 \end{equation}
Multiplying \eqref{penalized-NSCH}$_2$ by $u$ and \eqref{penalized-NSCH}$_3$ by $\mu$, integrating over $[0,L]\times[0,T]$ and adding them up, combining with  \eqref{the inequlity of f-lambda}, one has
\begin{eqnarray}\label{the basic energy inequality for density velocity and concentration difference}
&&\sup_{t\in[0,T]}\int_0^L\Big(\frac{1}{2}\rho u^2+\frac{\epsilon}{2}|\chi_x|^2+G(\rho)
+\frac{1}{4\epsilon}\rho\big(\chi^2-1\big)^2\Big)dx+\int_0^T\int_0^L\big(|\mu_x|^2+\nu|u_x|^2\big)dx\notag\\
&&\quad\leq\int_0^L\Big(\frac{1}{2}\rho_0 u_0^2+\frac{\epsilon}{2}|\chi_{0x}|^2+G(\rho_0)
+\frac{1}{4\epsilon}\rho_0\big(\chi_0^2-1\big)^2\Big)dx.
\end{eqnarray}
Multiplying \eqref{penalized-NSCH}$_3$ by $\chi$, integrating over $[0,L]$, combining with \eqref{the derivative formula of f-lambda} and \eqref{the inequality-1 of f-lambda}, one gets
\begin{equation}\label{the basic energy equality-2 for  concentration difference}
  \frac{1}{2}\frac{d}{dt}\int_0^L\rho\chi^2dx+\frac{1}{\epsilon}\int_0^L(3\chi^2-1)\chi_x^2dx+\epsilon\int_0^L\frac{1}{\rho}\chi_{xx}^2dx\leq0,
\end{equation}
by using \eqref{bound of density and chi}, adding \eqref{the basic energy inequality for density velocity and concentration difference} up, integrating over $[0,T]$, one has
\begin{eqnarray}\label{the basic energy inequality-2 for density velocity and concentration difference}
\sup_{[0,T]}\| \sqrt{\rho}(\chi(t)-\bar\chi)\|^2+\int_0^T\big(\| \chi_{x}(\tau)\|^2+\|\frac{1}{\sqrt{\rho}}\chi_{xx}(\tau)\|^2\big)d\tau\leq C\|\sqrt{\rho_0}(\chi_0-\bar{\chi})\|^2.
\end{eqnarray}
Noting that
\begin{equation}\label{identical relation}
 u_{xx}=-\big[\frac{1}{\rho}\big(\rho_t+\rho_x u\big)\big]_x=\big[\rho(\frac{1}{\rho})_t+\rho u(\frac{1}{\rho})_x\big]_x=\rho\frac{d}{dt}(\frac{1}{\rho})_x+\rho u(\frac{1}{\rho})_{xx},
\end{equation}
then \eqref{penalized-NSCH}$_2$ can be writeen as
\begin{equation}\label{the other form for NSCH-2}
  (\rho u)_t+(\rho u^2)_x+p'(\rho)\rho_x=\nu\big[\rho\frac{d}{dt}(\frac{1}{\rho})_x+\rho u(\frac{1}{\rho})_{xx}\big]-\frac{\epsilon}{2}\big(\chi_x^2\big)_x,
\end{equation}
Multiplying  \eqref{the other form for NSCH-2} by $(\frac{1}{\rho})_x$, integrating over $[0,L]$, one has
\begin{eqnarray}\label{the basic energy equality-2 for density}
&&\frac{d}{dt}\int_0^L\big(\frac\nu2\rho\big|\big(\frac{1}{\rho}\big)_x\big|^2-\rho u(\frac{1}{\rho})_x \big)dx+\int_0^L \frac{p'(\rho)}{\rho^2}\rho_x^2dx\notag\\
&&=-\int_0^L\rho u(\frac1\rho)_{xt}dx+\int_0^L(\rho u^2)_x(\frac1\rho)_xdx+\frac{\epsilon}{2}\int_0^L\big(\chi_x^2\big)_x(\frac{1}{\rho})_xdx\notag\\
&&=\int_0^L\Big((\rho u)_x(-\frac{\rho_t}{\rho^2})+(\rho u^2)_x(-\frac{\rho_x}{\rho^2})\Big)dx+\epsilon\int_0^L\chi_x\chi_{xx}(\frac1\rho)_xdx\notag\\
&&=\int_0^L u_x^2dx+\epsilon\int_0^L\chi_x\chi_{xx}(\frac1\rho)_xdx,
\end{eqnarray}
Multiplying \eqref{the basic energy equality-2 for density} by $\frac{\nu}{2}$, adding up to \eqref{the basic energy inequality for density velocity and concentration difference} and \eqref{the basic energy equality-2 for  concentration difference}, one obtains
\begin{eqnarray}\label{the basic energy equality-3 for density}
&&\frac{d}{dt}\int_0^L\Big[\frac{\nu^2}4\rho\big|\big(\frac{1}{\rho}\big)_x\big|^2-\frac{\nu}{2}\rho u(\frac{1}{\rho})_x +\frac{1}{2}\rho u^2+\rho \chi^2
+\frac{\epsilon}{2}|\chi_x|^2+G(\rho)
+\frac{\rho}{4\epsilon}(\chi^2-1)^2\Big]dx\notag\\
&&+\frac{\nu}{2}\int p'_\rho(\rho)\frac{\rho_x^2}{\rho^2}dx+\frac{\nu}{2}\int_0^L|u_x|^2dx+\int_0^L|\mu_x|^2dx+\frac{1}{\epsilon}\int_0^L(3\chi^2-1)\chi_x^2dx+\epsilon\int_0^L\frac{1}{\rho}\chi_{xx}^2dx\notag\\
&&
=\frac{\nu\epsilon}{2}\int_0^L\chi_x\frac{\chi_{xx}}{\rho^{\frac12}}\rho^{\frac12}\big(\frac{1}{\rho}\big)_xdx\leq \frac{\nu\epsilon}{\sqrt{2}}\|\chi_x\|^{\frac12}\|\chi_{xx}\|^{\frac12}\int_0^L|\frac{\chi_{xx}}{\rho^{\frac12}}|\big|\rho^{\frac12}\big(\frac{1}{\rho}\big)_x\big|dx\notag
\\
&&\leq2\nu\epsilon\|\chi_x\|\|\chi_{xx}\|+\frac{\nu^2\epsilon}{4}\Big\|\frac{\chi_{xx}}{\rho^{\frac12}}\Big\|^2\Big\|\rho^{\frac12}\big(\frac{1}{\rho}\big)_x\Big\|^2\notag\\
&&\leq4\nu^2\epsilon\|\chi_x\|^2
+\epsilon\|\rho^{\frac12}\|^2\|\rho^{-\frac12}\chi_{xx}\|^2+\frac{\nu^2\epsilon}{4}\Big\|\frac{\chi_{xx}}{\rho^{\frac12}}\Big\|^2
\Big\|\rho^{\frac12}\big(\frac{1}{\rho}\big)_x\Big\|^2,
\end{eqnarray}
combining with \eqref{bound of density and chi}, \eqref{the basic energy inequality for density velocity and concentration difference}, \eqref{the basic energy inequality-2 for density velocity and concentration difference},\eqref{the basic energy equality-3 for density},  one gets
\begin{eqnarray}\label{the first energy inequality}
 \left.\begin{array}{llll}
\displaystyle\int_0^L\Big(\rho\big|\big(\frac{1}{\rho}\big)_x\big|^2+\rho u^2+\rho\chi^2
+|\chi_x|^2+G(\rho)
+\rho (\chi^2-1)^2\Big)dx\\
\displaystyle+\int_0^T\int_0^L \Big(p'_\rho(\rho)\frac{\rho_x^2}{\rho^2}+|u_x|^2+|\mu_x|^2+(3\chi^2-1)\chi_x^2+\frac{1}{\rho}\chi_{xx}^2\Big)dxdt\\
\displaystyle \leq C\Big(\|(\rho_0,\chi_0)\|^2_1+\|\sqrt{\rho_0}u_0\|^2\Big).
\end{array}
\right.
\end{eqnarray}
Now we're going to give the upper and lower bounds of density. According to the method in Kanel'\cite{Kanel1968}, constructing a function
\begin{equation}\label{Kanel function-1}
\Psi(\rho):=\int_{\bar \rho}^\rho\frac{\sqrt{G(s)}}{s^{\frac32}}ds.
\end{equation}
By using the definition of $\bar\rho$ \eqref{average value}, we know that for fixed $t$, there exist a point $x_0\in[0,L]$, $\rho(x_0,t)=\bar\rho$, and one has
\begin{eqnarray}\label{the estimate of Psi}
|\Psi(\rho(x,t))|&=&\Big|\int_{x_0}^x\frac{\partial}{\partial s}\Psi(\rho(s,t))ds\Big|\leq\int_0^L\Big|\frac{G(\rho)^\frac12}{\rho^\frac32}\rho_x\Big|dx\notag\\
&\leq& \Big(\int_0^L G(\rho)dx\Big)^\frac12\Big(\int_0^L\rho\big(\frac{1}{\rho}\big)_x^2dx\Big)^\frac12\leq C(\rho_0,u_0,\chi_0).
\end{eqnarray}
On the other hand, noting that the definition of \eqref{the function G}, one gets
\begin{equation}\label{G-Kanel function}
G(\rho):=
\left\{\begin{array}{llll}
\displaystyle a\Big(\rho\ln\rho-\rho(\ln\bar\rho+1)+\bar\rho\Big)
\left\{\begin{array}{llll}
\thicksim\rho\ln\rho,&\rho\rightarrow 0^+,\\
\thicksim\rho\ln\rho,&\rho\rightarrow +\infty,
\end{array}\right.
\ \ &\gamma=1,\\
\displaystyle a\Big(\frac{\rho^\gamma}{\gamma-1}-\frac{\gamma\bar\rho^{\gamma-1}}{\gamma-1}\rho+\bar\rho^\gamma\Big)
\left\{\begin{array}{llll}
\thicksim\rho,&\rho\rightarrow 0^+,\\
\thicksim\rho^\gamma,&\rho\rightarrow +\infty,
\end{array}\right.
\ \ &\gamma>1,
 \end{array}\right.
\end{equation}
and then
\begin{equation}\label{limit of Kanel function}
\Psi(\rho):=\int_{\bar \rho}^\rho\frac{\sqrt{G(s)}}{s^{\frac32}}ds
 \left\{\begin{array}{llll}
\displaystyle\longrightarrow+\infty,\ &\rho\rightarrow+\infty,\\
 \displaystyle\longrightarrow-\infty,\ &\rho\rightarrow0^+.\end{array}\right.
\end{equation}
Taking advantage of \eqref{the estimate of Psi},   there exist two positive constants $\underline{m}>0$ and $\overline{M}>0$, satisfy
\begin{equation}\label{the lower and upper bounds of density}
 0<\underline{m}\leq\rho\leq\bar{M}<+\infty.
\end{equation}
From \eqref{penalized-NSCH}$_4$, one derives that
\begin{equation}\label{mu}
\frac{\epsilon}{\rho}\chi_{xx}=\frac{1}{\epsilon}\frac{\partial f_\lambda}{\partial\chi}(\chi)-\mu,
\end{equation}
differentiating \eqref{mu} with respect to $x$, one has
\begin{equation}\label{the third derivative for chi}
\epsilon\chi_{xxx}=-\rho\mu_x+\frac{\epsilon\rho_x}{\rho}\chi_{xx}+\frac{\rho}{\epsilon}\frac{\partial^2 f_\lambda}{\partial\chi^2}\chi_x,
\end{equation}
then directly from \eqref{the first energy inequality} and \eqref{the lower and upper bounds of density}, one gets at once
\begin{equation}\label{the third derivative of x for chi}
\int_0^T\|\chi_{xxx}\|^2dt\leq C\Big(\|(\rho_0,\chi_0)\|^2_1+\|u_0\|^2\Big).
\end{equation}
multiplying \eqref{penalized-NSCH}$_2$  by $-\frac{1}{\rho}u_{xx}$, integrating over $\mathbb{R}$ by parts, one gets
\begin{eqnarray}\label{the second derivative of velocity}
 &&\Big(\int\frac{1}{2}u^2_xdx\Big)_t+\nu\int_0^L\frac{1}{\rho}u_{xx}^2dx=\int_0^L u_{xx}\Big(\frac{p'_\rho}{\rho}\rho_x+uu_x+\frac{\epsilon}{\rho}\chi_x\chi_{xx}\Big)dx\notag\\
 &&\leq \frac{\nu}{4}\int_0^L\frac{1}{\rho}u_{xx}^2dx+\frac{c_1}{\nu}\Big(\int_0^L (p'_\rho)^2\frac{\rho_x^2}{\rho}dx+\int\rho u^2u_x^2dx+\int\frac{1}{\rho}\chi_x^2\chi_{xx}^2dx
 \Big)\notag\\
 &&\leq\frac{\nu}{4}\int_0^L\frac{1}{\rho}u_{xx}^2dx+\frac{c_1}{\nu}\Big(\int_0^L (p'_\rho)^2\frac{\rho_x^2}{\rho}dx+\|u_x\|\|u_{xx}\|\int_0^L\rho u^2dx+\|\chi_{x}\|\|\chi_{xx}\|\int_0^L\frac{1}{\rho}\chi^2_{xx}dx\Big)\notag\\
 &&\leq\frac{\nu}{2}\int_0^L\frac{1}{\rho}u_{xx}^2dx+\frac{c_2}{\nu}\Big(\int_0^L p'_\rho(\rho)\frac{\rho_x^2}{\rho^2}dx+\|u_x\|^2+\int_0^L\frac{1}{\rho}\chi^2_{xx}dx\Big),
\end{eqnarray}
integrating over $[0,T]$, from \eqref{the first energy inequality} and \eqref{the third derivative of x for chi}, one has
\begin{eqnarray}\label{the second derivative of velocity-2}
 \sup_{[0,T]}\|u_x(t)\|^2+\int_0^T\|u_{xx}(t)\|^2dt\leq C\Big(\|(\rho_0,u_0)\|^2_1+\|\chi_0\|^2_2\Big).
\end{eqnarray}
Multiplying \eqref{penalized-NSCH}$_3$ by $\chi_t$, integrating over $[0,L]$, one gets
\begin{eqnarray}\label{the derivative of t for chi}
  &&\int_0^L\rho\chi_t^2dx+\int_0^L\rho u\chi_x\chi_tdx=\int_0^L\mu\chi_{xxt}dx,\notag\\
  &&=-\frac{1}{\epsilon}\int_0^L\frac{\partial^2 f_\lambda}{\partial\chi^2}\chi_x\chi_{xt}dx-\frac{\epsilon }{2}\frac{d}{dt}\int_0^L\frac{1}{\rho}\chi_{xx}^2dx+\epsilon\int_0^L(\frac{1}{\rho})_t\chi_{xx}^2dx,
\end{eqnarray}
and consequently
\begin{eqnarray}\label{the derivative of t for chi-1}
  &&\int_0^L\rho\chi_t^2dx+\frac{1}{2\epsilon}\frac{d}{dt}\int_0^L\frac{\partial^2 f_\lambda}{\partial\chi^2}\chi_x^2dx+\frac{\epsilon }{2}\frac{d}{dt}\int_0^L\frac{1}{\rho}\chi_{xx}^2dx\notag\\
  &&=-\int_0^L\rho u\chi_x\chi_tdx+\frac{\epsilon}{2}\int_0^L(\frac{1}{\rho})_t\chi_{xx}^2dx\notag\\
  &&\leq\frac{1}{2}\int_0^L\rho\chi_t^2dx+ C\|\chi_x\|\|\chi_{xx}\|\int_0^L\rho u^2dx+\frac{\epsilon}{2}\Big|\int_0^L\big(\frac{u_x}{\rho}\chi_{xx}^2+\frac{\rho_x u}{\rho^2}\chi_{xx}^2\big)dx\Big|\notag\\
  &&\leq\frac{1}{2}\int_0^L\rho\chi_t^2dx+ C\|\chi_x\|\|\chi_{xx}\|+C\|\chi_{xx}\|^{\frac12}\|\chi_{xxx}\|^{\frac12}\Big(\int_0^L u_x^2dx+\int_0^L\chi_{xx}^2dx\Big)\notag\\
  &&\ \ +C\|\chi_{xx}\|^{\frac12}\|\chi_{xxx}\|^{\frac12}\|u\|^{\frac12}\|u_x\|^{\frac12}\big(\int_0^L\rho_x^2dx+\int_0^L\chi_{xx}^2dx\big)\notag\\
  &&\leq C\Big(\|\chi_x\|^2+\|\chi_{xx}\|^2+\|\chi_{xx}\|^{\frac12}\|\chi_{xxx}\|^{\frac12}\big(\int_0^L u_x^2dx+\int_0^L\rho_x^2dx+\int_0^L\chi_{xx}^2dx\big)\Big)\notag\\
  &&\qquad +\frac{1}{2}\int_0^L\rho\chi_t^2dx,
\end{eqnarray}
that is, combining with  \eqref{the first energy inequality}, \eqref{the third derivative of x for chi} and \eqref{the second derivative formula of f-lambda}, one obtains
\begin{eqnarray}\label{the inequality about derivative of t for chi}
\sup_{[0,T]}\Big(\|\chi_x(t)\|^2+\|\chi_{xx}(t)\|^2\Big)+\int_0^T\|\chi_t(t)\|^2dt\leq C\|\chi_0\|^2_2.
\end{eqnarray}
Differentiating \eqref{penalized-NSCH}$_3$ with respect of $t$, one has
\begin{equation}\label{the higher order derivatives of t for chi}
\rho\chi_{tt}+\rho_t\chi_t+\rho_t u\chi_x+\rho u_t\chi_x+\rho u\chi_{xt}=\mu_{xxt},
\end{equation}
multiplying \eqref{the higher order derivatives of t for chi} by $\chi_t$, integrating over $[0,L]$, one gets
\begin{eqnarray}\label{the energy equality of the higher order derivatives of t for chi}
&&\frac{1}{2}\frac{d}{dt}\int_0^L\rho\chi_t^2dx+\int_0^L\frac{1}{\rho}\chi_{xxt}^2dx\notag\\
&&=-2\int_0^L\rho u\chi_t\chi_{xt}dx+\int_0^L\rho_xu^2\chi_t\chi_xdx+\int_0^L\rho u_xu \chi_x\chi_tdx\notag\\
\!\!\!\!\!\!&&-\int_0^L\rho u_t\chi_x\chi_tdx+\frac{1}{\epsilon}\int_0^L\frac{\partial^2 f_\lambda}{\partial \chi^2}\chi_t\chi_{xxt}dx-\int_0^L\frac{u_x}{\rho}\chi_{xx}\chi_{xxt}dx-\int_0^L\frac{\rho_x}{\rho^2}u\chi_{xx}\chi_{xxt}dx\notag\\
\!\!\!\!\!\!&&\leq\frac{1}{4}\int_0^L\chi_{xxt}^2dx+\frac{1}{4}\int_0^L\rho u_t^2dx+C\int_0^L\big(\chi_t^2+\chi_{xx}^2\big)dx,
\end{eqnarray}
on the other hand, multiplying  \eqref{penalized-NSCH}$_2$ by $u_t$, integrating over $[0,L]$, one gets
\begin{eqnarray}\label{the energy equality of the higher order derivatives of t for velocity}
&&\frac{1}{2}\frac{d}{dt}\int_0^L u^2_xdx+\int_0^L\rho u_{t}^2dx=-\int_0^L\rho u u_x u_tdxdt-\int_0^L p'_\rho\rho_xu_tdxdt-\epsilon\int_0^L\chi_x\chi_{xx}u_tdxdt\notag\\
&&\leq \frac{1}{4}\int_0^L\rho u_t^2dx+C\int_0^L \big(u_x^2+\rho_x^2+\chi_{xx}^2\big)dx,
\end{eqnarray}
summing \eqref{the derivative of t for chi}, \eqref{the energy equality of the higher order derivatives of t for chi} and \eqref{the energy equality of the higher order derivatives of t for velocity} up,
making use of \eqref{the first energy inequality}, \eqref{the third derivative of x for chi}, \eqref{the second derivative of velocity-2} and \eqref{the second derivative formula of f-lambda},
 one obtains
\begin{eqnarray}\label{the energy inequality of the derivatives of t xxt for chi}
&&\sup_{[0,T]}\big(\|\chi_{x}(t)\|_1^2+\|\chi_t(t)\|^2+\|u_x(t)\|^2\big)+\int_0^T\big(\|\chi_t(t)\|^2+\|\chi_{xxt}^2(t)\|^2+\|u_t(t)\|^2\big)dx\notag\\
&&\leq C\Big(\|(\rho_0,u_0)\|^2_1+\|\chi_0\|^2_2+\|\chi_{0t}\|^2\Big).
\end{eqnarray}
Differentiating (\ref{penalized-NSCH})$_3$ with respect of $x$, multiplying by $\mu_{xxx}$, one gets
\begin{equation}\label{the derivative of t for density-1}
  \sup_{[0,T]}\|\chi_{xxx}(t)\|^2+\int_0^T\|\mu_{xxx}\|^2dx \leq  C\|\chi_0-\bar\chi\|^2_3.
\end{equation}
 Directly from \eqref{penalized-NSCH}, by using the estimates \eqref{the first energy inequality}, \eqref{the second derivative of velocity-2}, \eqref{the energy inequality of the derivatives of t xxt for chi}, one has
\begin{eqnarray}\label{the derivative of t for density}
  &&\sup_{[0,T]}\big(\|\rho_t(t)\|^2+\|\mu(t)\|^2+\|\mu_{xx}(t)\|^2\big)+\int_0^T\big(\|\mu_t(t)\|^2+\|\mu_{xxx}(t)\|^2\big)dt\notag\\
  &&\leq C\Big(\|(\rho_0,u_0)\|^2_1+\|\chi_0\|^2_2+\|\chi_{0t}\|^2\Big).
\end{eqnarray}
Differentiating (\ref{penalized-NSCH})$_2$ with respect of $t$, one derives
\begin{eqnarray}\label{the t derivative of u}
  &&\rho_tu_t+\rho u_{tt}+\rho_t uu_x+\rho u_t u_x+\rho u u_{xt}+p_{xt}=\nu u_{xxt}-\epsilon\chi_{xt}\chi_{xx}-\epsilon\chi_{x}\chi_{xxt}.
\end{eqnarray}
Multiplying \eqref{the t derivative of u} by $u_t$, integrating over $[0,L]\times[0,T]$ by parts, making use of \eqref{the first energy inequality}, \eqref{the third derivative of x for chi}, \eqref{the second derivative of velocity-2} and \eqref{the energy inequality of the derivatives of t xxt for chi}, one gets
\begin{eqnarray}\label{the energy equality for t derivative for u}
\sup_{[0,T]}\| u_t(t)\|^2+\int_0^T\int_0^L u_{xt}^2dxdt\leq C\Big(\|(\rho_0-\bar{\rho},u_0)\|^2_1+\|\chi_0-\bar\chi\|^2_2+\|\chi_{0t}\|^2+\|u_{0t}\|^2\Big).
\end{eqnarray}
Differentiating (\ref{penalized-NSCH})$_2$ with respect of $x$, one derives
\begin{eqnarray}\label{the x derivative of u}
 \rho_xu_t+\rho u_{xt}+\rho_x uu_x+\rho u_x u_x+\rho u u_{xx}+p_{xx}=\nu u_{xxx}-\epsilon\chi_{xx}^2-\epsilon\chi_{x}\chi_{xxx}.
\end{eqnarray}
Multiplying \eqref{the x derivative of u} by $u_{xxx}$, integrating over $[0,L]\times[0,T]$ by parts, making use of \eqref{the first energy inequality}, \eqref{the third derivative of x for chi}, \eqref{the second derivative of velocity-2} and \eqref{the energy inequality of the derivatives of t xxt for chi}, \eqref{the energy equality for t derivative for u}, one gets
\begin{eqnarray}\label{the energy equality for x derivative for u}
\sup_{[0,T]}\| u_{xx}(t)\|^2+\int_0^T\int_0^L u_{xxx}^2dxdt\leq C\Big(\|\rho_0\|^2_1+\|(u_0,\chi_0)\|^2_2+\|\chi_{0t}\|^2+\|u_{0t}\|^2\Big).
\end{eqnarray}
Differentiating \eqref{the other form for NSCH-2} with respect of $x$, one derives
\begin{equation}\label{the xx derivative of density}
(\rho u)_{xt}+(\rho u^2)_{xx}+p_{xx}=\nu\big[\rho\frac{d}{dt}(\frac{1}{\rho})_{x}+\rho u(\frac{1}{\rho})_{x}\big]_x-\frac{\epsilon}{2}\big(\chi_x^2\big)_{xx},
\end{equation}
Multiplying \eqref{the xx derivative of density} and \eqref{the x derivative of u} by $\big(\frac{1}{\rho}\big)_{xx}$, $u_x$ respectively,  integrating over $[0,L]$ and summing,  one obtains
\begin{eqnarray}\label{the second deritive for density}
&&\sup_{[0,T]}\Big(\|u_x(t)\|^2+\|\rho_{xx}(t)\|^2\Big)+\int_0^T\big(\|\rho_{xx}(\tau)\|^2+\|u_{xx}(\tau)\|^2\big)d\tau\notag\\
&& \leq C\Big(\|\rho_0\|^2_2+\|u_0\|^2_1+\|\chi_0\|_3^2\Big).
\end{eqnarray}
Furthermore,  similar to the proof above,  one has
\begin{equation}\label{the forth derivative of x for chi}
  \sup_{[0,T]}\|\chi_{xxxx}(t)\|^2+\int_0^T\|\mu_{xxxx}\|^2dx \leq  C\|\chi_0-\bar\chi\|^2_4,
\end{equation}
Now we will let $\lambda\rightarrow 0$ in approximate problem \eqref{penalized-NSCH}.
For $0<\lambda<1$, from the results above,  there exists a solution   $(\rho_\lambda,u_\lambda,\chi_\lambda)$ of approximate problem  \eqref{penalized-NSCH}, satisfies
$(\rho_\lambda,u_\lambda,\chi_\lambda)\in X_{0,+\infty,B}([0,+\infty))$.
Multiplying \eqref{penalized-NSCH}$_4$ by $\beta_\lambda(\chi_\lambda)$, integrating over $[0,L]$, one has
\begin{eqnarray}\label{the energy equality of beta-lambda}
  &&\epsilon\int_0^L\rho_\lambda(\chi_\lambda)_x(\beta_\lambda)_xdx+\frac{1}{\epsilon\lambda}\int_0^L\rho_\lambda\beta_\lambda^2dx\notag\\
  &&=\int_0^L\rho_\lambda\mu_\lambda\beta_\lambda dx
  -\int_0^L\frac{\rho_\lambda}{\epsilon}(\chi_\lambda^3-\chi_\lambda)\beta_\lambda dx,\notag\\
  &&\leq C\epsilon\lambda\Big(\int_0^L\rho_\lambda\mu_\lambda^2dx+\frac{1}{\epsilon^2}\int_0^L\rho_\lambda\Big(\frac{\partial f_\lambda}{\partial\chi}(\rho_\lambda,\chi_\lambda)\Big)dx\Big)+\frac{1}{2\epsilon\lambda}\int_0^L\rho_\lambda\beta_\lambda^2dx,
\end{eqnarray}
combining with
\begin{eqnarray}\label{the derivative of beta-lambda}
  \epsilon\int_0^L\rho_\lambda(\chi_\lambda)_x(\beta_\lambda)_xdx&=&\epsilon\int_0^L\rho_\lambda\beta'_\lambda(\chi_\lambda)^2_x dx\notag\\
  &\geq&\epsilon\int_0^L\rho_\lambda(\beta'_\lambda)^2(\chi_\lambda)^2_x dx\geq\epsilon\int_0^L\rho_\lambda(\beta_\lambda)_x^2dx,
\end{eqnarray}
one obtains
\begin{equation}\label{the estimate for beta-lambda}
  \|\beta_\lambda\|_{L^2(0,T;L^2)}\leq C\lambda,\  \textrm{and}\   \|(\beta_\lambda)_x\|_{L^2(0,T;L^2)}\leq C\lambda^\frac12,
\end{equation}
thus, letting $\lambda\rightarrow 0$,  one has $\beta(\chi)=0$,
that is,
\begin{equation}\label{the lower and upper bound of chi}
  -1\leq\chi\leq1.
\end{equation}
Moreover,  by using the compactness theory and the Lions-Aubin argument, combining with the results above,
it is easy to know that the  system \eqref{Euler-NSCH}-\eqref{periodic boundary} has a  strong solution $(\rho,u,\chi)$ when $\lambda\rightarrow 0$.
Therefore, the proof of Proposition \ref{a priori estimate proposition} is completed.
\end{proof}

 Now we will give the proof of  Theorem \ref{main thm-1}. For the global existence,  we only need to get rid of the small condition \eqref{small condition for chi} in the  Proposition \ref{a priori estimate proposition}.  Noting that the small condition for the initial value $\chi$ only used in \eqref{the basic energy equality-2 for  concentration difference} to get the estimate \eqref{the basic energy inequality-2 for density velocity and concentration difference}, rewriting   \eqref{the basic energy equality-2 for  concentration difference} as follows
 \begin{equation}\label{the basic energy equality-2 for  concentration difference without small condition}
  \frac{1}{2}\frac{d}{dt}\int_0^L\rho\chi^2dx+\frac{1}{\epsilon}\int_0^L3\chi^2\chi_x^2dx+\epsilon\int_0^L\frac{1}{\rho}\chi_{xx}^2dx\leq\frac{1}{\epsilon}\int_0^L\chi_x^2dx,
\end{equation}
integrating over $[0,T]$, one has
 \begin{equation}\label{the basic energy estimate-2 for  concentration difference without small condition}
  \frac{1}{2}\int_0^L\rho\chi^2(t)dx+\frac{1}{\epsilon}\int_0^L3\chi^2\chi_x^2dx+\epsilon\int_0^T\int_0^L\frac{1}{\rho}\chi_{xx}^2dx=\frac{1}{\epsilon}\int_0^T\int_0^L\chi_x^2dx+\frac{1}{2}\int_0^L\rho\chi_0^2dx,
\end{equation}
by using \eqref{the basic energy inequality for density velocity and concentration difference} and Gronwall's inequality, one obtains similar energy estimates  without the small condition \eqref{small condition for chi}, but the constant $C$ depending on $T$, and the global existence for \eqref{Euler-NSCH}-\eqref{periodic boundary} is obtained, the uniqueness of this solution can be obtained by the classical method, we omit it. At last, it suffices to present the large time behavior for the solution $(\rho,u,\chi)$ of system \eqref{Euler-NSCH}-\eqref{periodic boundary}. By using the energy inequality \eqref{a priori estimate}, combining with Sobolev embedding theorem for periodic functions, one has at once
\begin{eqnarray}\label{the large time behavior}
 \sup_{x\in[0,L]}\Big|(\rho-\bar\rho,\rho u-\overline{\rho u},\rho\chi-\overline{\rho\chi})(x,t)\Big|&\leq& L \|(\rho_x,(\rho u)_x,(\rho\chi)_x)(t)\|\notag\\
  &\longrightarrow&0,  \ \ \mathrm{as}\ t\rightarrow+\infty,
\end{eqnarray} thus, \eqref{asymptotic stability} is obtained.
Moreover, by the similar way above, we have the proof of Theorem \ref{main thm-2}, the details are omitted here.

\

\noindent {\textbf{Acknowledgments:}  The research of  X. Shi was  partially supported by National Natural Sciences Foundation of China No. 11671027 and 11471321. The research of M. Mei was supported in part by NSERC 354724-2016 and FRQNT grant 256440.}

\end{document}